\documentclass{amsart}

\newtheorem{theorem}[equation]{Theorem}
\newtheorem{proposition}[equation]{Proposition}
\newtheorem{lemma}[equation]{Lemma}

\newtheorem{corollary}[equation]{Corollary}

\newtheorem{claim}[equation]{Claim}

\theoremstyle{remark}
\newtheorem{remark}[equation]{Remark}
\theoremstyle{definition}
\newtheorem{definition}[equation]{Definition}

\numberwithin{equation}{section}

 \renewcommand{\qed}{\hspace*{\fill} \setlength{\unitlength}{1mm}
 \begin{picture}(2.5,2.5)
       \put(0,0){\framebox(2.5,2.5){}}
   \end{picture}
 \setlength{\unitlength}{1pt}}

\newcommand{\eps}{\varepsilon}
\newcommand{\vphi}{\varphi}

\newcommand{\me}{\mathrm{e}}
\newcommand{\side}{\operatorname{side}}
\newcommand{\Vol}{\operatorname{Vol}}

\newcommand{\supp}{\operatorname{Supp}}

\newcommand{\dist}{\operatorname{dist}}

\newcommand{\length}{\operatorname{length}}

\newcommand{\Av}{\mathrm{Av}}
\newcommand{\sgn}{\mathrm{sgn}}

\newcommand{\Nod}{\mathsf{Nod}}

\newcommand{\Dl}{\mathcal{D}}

\newcommand{\dt}{\;dt}
\newcommand{\dx}{\;dx}
\newcommand{\dy}{\;dy}

\newcommand{\natls}{\mathbb{N}}
\newcommand{\ratls}{\mathbb{Q}}

\newcommand{\bx}{{\bf x}}
\newcommand{\by}{{\bf y}}
\newcommand{\vol}{{\rm vol}}
\newcommand{\ar}{{\rm area}}
\newcommand{\bd}{{\partial}}

\newcommand{\CC}{\mathbb{C}}

\newcommand{\QQ}{\mathbb{Q}}
\newcommand{\RR}{\mathbb{R}}

\newcommand{\cN}{{\mathcal{N}}}

\newcommand{\cV}{{\mathcal{V}}}

\newcommand{\cF}{{\mathcal{F}}}
\newcommand{\cH}{{\mathcal{H}}}

\newcommand{\cA}{{\mathcal{A}}}

\newcommand{\cD}{{\mathcal{D}}}

\newcommand{\tld}{\tilde}

\begin{document}

\title[Tubular Neighborhoods of Nodal Sets]
{Tubular Neighborhoods of Nodal Sets and Diophantine Approximation}

\author[Dmitry Jakobson]{Dmitry Jakobson}
\address{Department of Mathematics and
Statistics, McGill University, 805 Sherbrooke Str. West,
Montr\'eal QC H3A 2K6, Ca\-na\-da.}
\email{jakobson@math.mcgill.ca}

\author[Dan Mangoubi]{Dan Mangoubi}
\address{IH\'ES, Le Bois-Marie, 35, Route de Chartres,
F-91440 Bures-sur-Yvette,
France}
\email{mangoubi@ihes.fr}

%

\thanks{ The first author was supported by NSERC, NATEQ and Dawson
fellowship. The second author was supported by EPDI and CRM-ISM fellowship.}

\begin{abstract}
We give upper and lower bounds on the volume of a tubular
neighborhood of the nodal set of an eigenfunction of the Laplacian
on a real analytic closed Riemannian manifold $M$. As an
application we consider the question of approximating points on
$M$ by nodal sets, and explore analogy with approximation by
rational numbers.
\end{abstract}
\maketitle


\section{Introduction and Main Results}
Let $(M, g)$ be a real analytic closed Riemannian manifold.
In the first part of this paper we give upper and lower bounds on the volume of
tubular neighborhoods of nodal sets.
Consider the eigenequation
$$\Delta \phi_\mu + \mu^2\phi_\mu =0, $$
where $\Delta$ is the Laplace--Beltrami operator on $M$.
We denote the nodal set $\{\phi_\mu=0\}$ by $\cN_\mu$.
 Consider the tubular neighborhood of the nodal set
\begin{equation}
 T_{\mu, \delta} = \{ x\in M\, :\, \dist(x, \cN_{\mu}) < \delta\}\ .
\end{equation}
We prove
\begin{theorem}
\label{thm:nodal-tube}
  Let $(M, g)$ be a real analytic closed Riemannian manifold.
  Then there exist $C_1, C_2, C_3>0$ such that
  $$C_1\mu\delta \leq \Vol(T_{\mu, \delta}) \leq C_2\mu\delta ,$$
  whenever $\mu\delta \leq C_3$.
\end{theorem}

To put Theorem~\ref{thm:nodal-tube} in the right
context we recall
\begin{theorem}[{\cite[Theorem~1.2]{DF}}]
\label{thm:nodal-volume} Let $(M, g)$ be a closed real analytic
Riemannian manifold. Then, there exist $C_4,C_5>0$ such that
  $C_4\mu\leq \Vol_{n-1} (\cN_\mu) \leq C_5\mu,$
  where $\Vol_{n-1}$ is the $(n-1)$-dimensional Hausdorff measure on $M$.
\end{theorem}

From this perspective, we see that
Theorem~\ref{thm:nodal-tube} describes a regularity property of
the nodal set. For example, the upper bound implies that the nodal
set does not have too many needles or very narrow branches, while
the lower bound says that the nodal set doesn't curve too much.

For the proof of Theorem~\ref{thm:nodal-tube}
we need to study the behavior of eigenfunctions in
all scales $0<\delta\leq 1/\mu$ ($1/\mu$ is called the wavelength).
Roughly, we show that for most points $x$, $\phi_\mu(x)$
is comparable to the average of $\phi_\mu$ on a ball 
of radius~$\delta$ centered at $x$.
This study is the content of
Sections~\ref{sec:hol-dim-1}-\ref{sec:refinements},
and it extends
the work of Donnelly and Fefferman
in Section~5 of~\cite{DF}, where they consider 
scales \emph{comparable} to the wavelength 
$C_1/\mu\leq \delta\leq C_2/\mu$.

Donnelly and Fefferman showed that by considering
an analytic continuation of $\phi_\mu$, one can treat our problem
by studying polynomials in dimension one, and then applying an induction
argument. We adopt this approach here.
The key proposition is Proposition~\ref{prop:eigen-small-scale-v2}. 
Most of its proof
goes without change from the proof of Proposition~5.11
in~\cite{DF}. We had to adjust the arguments from~\cite{DF} in two
main points. The first is the proof in dimension one,
where we added the parameter $\delta$ to the proofs in~\cite{DF},
and showed that everything goes through.
The second is in
the proof of Proposition~\ref{prop:control-quotient} where the
change of variables argument is different and more subtle than the parallel
argument in~\cite{DF}.

The proof of the lower bound in Theorem~\ref{thm:nodal-tube} is
given in Section~\ref{sec:lower}. It is based on the behavior of
eigenfunctions in scales \emph{comparable} to the wavelength and 
on the Brunn-Minkowski inequality. 

The idea of the proof of the upper bound in Theorem~\ref{thm:nodal-tube} 
was suggested to the authors by C.~Fefferman. We give the proof
in Section~\ref{sec:upper}. A proof of the upper
bound by different methods can be found in~\cite{yomdin}.
Our proof is based on the upper bound in Theorem~\ref{thm:nodal-volume}
and our study of eigenfunctions in all scales $\delta\leq C/\mu$.


In Section~\ref{sec:dim2} we consider the special case where
$\dim(M)=2$. We show that the lower bound is true for any
\emph{smooth} surface and the upper bound is true for any smooth
surface which satisfies Yau's conjecture.

 In the second
part of the paper we make an attempt to look simultaneously on the
ensemble of nodal sets which belong to different eigenvalues.
Consider first a simple example: Eigenfunctions on $M=[0,\pi]$
with the standard metric and (say) with Dirichlet boundary
conditions.  Then
$$
\mu_k=k,\ \phi_k(x)=\sin(kx),\ \cN_k=\left\{\frac{\pi j}{k}:0\leq j\leq k
\right\}.
$$
Accordingly, the set $\cN_k$ is $\pi/(2k)$-dense in $M$.
Interestingly, a similar result holds on any smooth Riemannian manifold
(see e.g. \cite{Bru}):
\begin{proposition}\label{prop:dense}
There exists $C>0$ (which depends only on $M,g$) such that
$$B(x,C/\mu)\cap\cN(\phi_\mu)\neq\emptyset$$
for any $x\in M$ and $\mu>0$.
\end{proposition}
Here $B(x,r)$ denotes the ball of radius $r$ centered at $x\in M$.
Thus $\cN_\mu$ is $C/\mu$-dense in $M$.

To study the rate of approximation by $\cN_\mu$ as $\mu\to\infty$ in
more detail, consider again the case of $M=[0,\pi]$ where
approximating by points in $\cN_k$ is equivalent (after rescaling by
$\pi$) to approximating by rationals with denominator $k$. It is
well-known (see e.g.~\cite{K}) that the distance
from any $x\in[0,1]$ to the $m$-th convergent of its continued
fraction expansion $p_m/q_m$ is $O(1/q_m^2)$.  However, the {\em
denominator} $q_m$ of the $m$-th continued fraction grows
exponentially in $m$ for $x\notin\QQ$ (\cite{K}).

Denote by $||x||$ the distance from $x\in\RR$ to the nearest integer.
The following proposition can be found in~\cite{K} and is proved
by an application of the Borel-Cantelli Lemma.
\begin{proposition}\label{khinchin}
If $\sum_q\psi(q)$ converges, then for Lebesgue-almost all $x$,
there exist only finitely many $q$ such that $||qx||<\psi(q)$.
\end{proposition}

Taking $\psi(q)=C/q^{1+\eps}$ in Proposition \ref{khinchin}
 we conclude that
\begin{corollary}\label{ratls}
Given $C,\eps>0$, for Lebesgue-almost all $x\in[0,1]$ the inequality
$$
|x-p/q|<C/q^{2+\eps}
$$
has finitely many integer solutions $(p,q)$.
\end{corollary}
Equivalently, almost all $x\in M=[0,\pi]$ {\em cannot} be approximated by
points in $\cN_k$ to within $C/k^{2+\eps}$ infinitely often.
We prove an analogous statement for any real analytic manifold $M$.

To characterize the {\em rate} of approximation by nodal sets,
we make the following definition:
\begin{definition}\label{def:rate}
Given $b>0$ {\em (exponent)}, and $C>0$ {\em (constant)}, let
$M(b,C)$ be the set of all $x\in M$ such that there exists an
infinite sequence of eigenvalues $\mu_k\to\infty$ for which
$$
B\left(x,\frac{C}{\mu_k^b}\right)\cap\cN(\phi_{\mu_k})\neq\emptyset.
$$
\end{definition}
For example, Proposition \ref{prop:dense} implies that $M(1,C)=M$ for
some $C>0$.  Also, Corollary \ref{ratls} implies that for
$M=[0,\pi]$, we have $\Vol(M(2+\eps,C))=0\ \forall
C,\eps>0$.
We prove
\begin{theorem}
\label{thm:approx}
Let $(M, g)$ be a closed real analytic Riemannian manifold of dimension $n$.
Then for
any $C>0,\eps>0,$
$$
\Vol (M(n+1+\eps,C))=0.
$$
\end{theorem}
The proof consists of Theorem~\ref{thm:nodal-tube},
the Borel--Cantelli Lemma and Weyl's asymptotics of eigenvalues.
\subsection{A Reader's Guide}
In Sections~\ref{sec:hol-dim-1}-\ref{sec:eigen-small-scale}
we study eigenfunctions in small scales.
The key proposition is Proposition~\ref{prop:eigen-small-scale-v2},
which roughly shows that for most points $x$, $\phi_\mu(x)$ is comparable
to the average of $\phi_\mu$ on a ball of radius $\delta$.
On a first reading one may assume this proposition.
In Section~\ref{sec:good-bad} we show how 
Proposition~\ref{prop:eigen-small-scale-v2} implies
geometric information on the nodal set and its neighborhood.
Section~\ref{sec:refinements} is a technical 
section which helps us treat the scales $\delta=C_1/\mu$ with $C_1$
large. The results of Sections~\ref{sec:eigen-small-scale},~\ref{sec:good-bad}
and~\ref{sec:refinements} are combined 
in Section~\ref{sec:lower} in order to prove the lower bound in
Theorem~\ref{thm:nodal-tube}. 
Section~\ref{sec:upper} gives the line of proof of the
upper bound in Theorem~\ref{thm:nodal-tube}. On a first reading
one may start with this section and move to
sections~\ref{sec:eigen-small-scale} and~\ref{sec:good-bad} when necessary.
In Section~\ref{sec:approx} we combine the upper bound in
Theorem~\ref{thm:nodal-tube} with Weyl's Law and the
Borel-Cantelli Lemma in order to establish
Theorem~\ref{thm:approx}. In Section~\ref{sec:dim2} we discuss
Theorem~\ref{thm:nodal-tube} for smooth surfaces. In
Section~\ref{sec:discussion} we discuss possible extensions of the
approximation result.

\subsection{Acknowledgments} The idea of the proof of 
Theorem~\ref{thm:nodal-tube} was 
suggested to the authors by C. Fefferman.

The authors would like to thank an anonymous referee who helped to
find a gap in an earlier version of this paper, and a second referee
due to him the paper is in a much more readable form. 
The authors would like
to thank M.~Sodin for his motivating question about the lower
bound on the volume of a nodal tube.
The authors would also like to thank E.~Bogomolny, Y.~Fyodorov, J.~Marklof, I.~Polterovich, Z.~Rudnick, U.~Smilansky, J.~Toth, I.~Wigman, Y.~Yomdin 
and S.~Zelditch
for stimulating discussions about this problem.

The first author would also like to thank the organizers of the
Workshop on Wavefunctions (Univ. of Bristol, September 2005), and
the organizers of the Workshop on Dynamical Systems and Related
Topics in honor of Y. Sinai (Univ. of Maryland, march 2006) for
their hospitality. A large part of this paper was written while the
first author visited Max Planck Institute for Mathematics in Bonn,
Germany and the IHES, France;  their hospitality is greatly
appreciated.

The second author is an EPDI fellow at the IHES, France;
Their hospitality and support is gratefully acknowledged.
The second author would also like to thank McGill
University and the CRM, Montreal where this work begun
for their hospitality during his stay in Montreal.
\section{Holomorphic Functions in Small Scales - dimension 1}
\label{sec:hol-dim-1}
In this section we describe the behavior
of holomorphic functions of one variable in small scales.
The proofs in this section follow closely the proofs in section~5 of~\cite{DF}.

We denote by $B_r \subset \CC$ the disk $|z|\leq r$. Suppose $F$ is
holomorphic on $B_3$ and satisfies the following growth assumption:
  \begin{equation}
    \label{growth}
    \sup_{B_2}|F|\leq |F(0)|\me^{C_1\mu}\ .
  \end{equation}

Let $I\subset\RR$ denote the interval $[-1, 1]$. 
Let $0<\delta<1/\mu$ be given.
We decompose $I$ into disjoint subintervals of sizes
$\delta<|I_\nu|<2\delta$. 
Given $x\in I$, we denote by $I_x$ the 
subinterval to which $x$ belongs ($I_x$ is
defined outside a set of measure $0$). 
We denote by $\Av_{I_{x}} F$ the average of $F$ on $I_{x}$.
The main proposition of this
section is
  \begin{proposition}
   \label{prop:wildintervals}
   Assume $F$ satisfies~(\ref{growth}).
   For all $\eps>0$ there exists a subset 
   $E_\eps\subseteq I$ of measure $|E_\eps|\leq C_2\eps\mu\delta$
   such that 
   $$\frac{1}{C_3(\eps)}\leq \frac{|F(x)|}{\Av_{I_x} |F|} \leq C_3(\eps)
   \quad \forall x\in I\setminus E_\eps\ ,$$
   with $C_3(\eps) =\me^{11/\eps^2}$.
  \end{proposition}

Proposition~\ref{prop:wildintervals} generalizes Proposition~5.1
from~\cite{DF}. The main new point here is the introduction of the
parameter $\delta$ of the subdivision, while in~\cite{DF} the size
of the subdivision is taken to be comparable to $1/\mu$. A minor
technical difference is that here we also allow subdivisions with
non-fixed size of the subintervals. This will serve us in the change
of variable argument in the proof of
Proposition~\ref{prop:control-quotient}.

The first step we make is a reduction to polynomials.
It is shown in Section~5 of~\cite{DF}
%
%
%
%
   \begin{lemma}[{\cite[Lemma 5.2]{DF}}]
     \label{lem:num-roots}
     $F$ has at most $C_4\mu$ zeroes in $B_{3/2}$.
   \end{lemma}
Denote the set of zeroes of $F$ in $B_r$ by $Z_r(F)$.
Fix $r<3/2$ close to $3/2$ so that $F$ does not have zeroes on $|z|=r$.
Let $P(z) := \prod_{\alpha\in Z_r(F)} (z-\alpha)$.
$P$ is a polynomial of degree $d\leq C_4\mu$.
Let $f(z) = \log |P(z)|$.
The next lemma shows that we can assume $F(z) = P(z)$.
\begin{lemma}
\begin{itemize}
\item[(i)]
$$\left|\left(\log|F(x)|-\log |F(y)|\right) 
- \left(f(x)-f(y)\right)\right| \leq C_5\quad\forall\nu\forall 
x,y\in I_\nu\ .$$
\item[(ii)] 
$$C_6 \frac{|P(x)|}{\Av_{I_x} |P|}
\leq\frac{|F(x)|}{\Av_{I_x} |F|}\leq 
C_7 \frac{|P(x)|}{\Av_{I_x} |P|} \ .$$
\end{itemize}
\end{lemma}
\begin{proof}
  Let $$ B_r(z, \alpha) = \frac{(z-\alpha)/r}{1-\bar{\alpha}z/r^2} $$
  be the Blaschke factor.
  We write
     $$F(z) = \me^{g(z)}\prod_{\alpha\in Z_r(F)} B_r(z, \alpha)\ .$$
\noindent We calculate
\begin{multline}
\label{log-calcul}
  \log|F(x)|-\log|F(y)| = \Re(g(x)-g(y)) + \\ (\log |P(x)|-\log|P(y)|) -
  \sum_{\alpha\in Z_r(F)} (\log|r^2-\alpha x| - \log|r^2-\alpha y|)\ .
\end{multline}
  The first term on the right hand side of~(\ref{log-calcul})
  is handled by Lemma~5.3~(iii) of~\cite{DF}:
  $$|\Re(g(x) - g(y))|\leq \max_{I_\nu}|\nabla \Re(g)| |x-y| 
\leq C_8\mu\delta\ .$$

  To bound the third term in the right hand side of~(\ref{log-calcul})
  one should only check by direct computation that
  $$\sup_{|x|\leq 1}|(\log|r^2-\alpha x|)'|=
    \sup_{|x|\leq 1}\left|
\left(\Re\frac{-\alpha}{r^2-\alpha x}\right)
\right|\leq \frac{1}{r-1}= C_9\ .$$

  The conclusion of part~(i) of the Lemma follows.

Part~(i) says that  
$$\forall y\in I_x,\ \me^{-C_5}\frac{|P(y)|}{|P(x)|}\leq  
\frac{|F(y)|}{|F(x)|} \leq \me^{C_5}\frac{|P(y)|}{|P(x)|}\ .$$
It only remains to integrate over $I_x$ in order to conclude part~(ii).
\end{proof}
%

We now turn to bound $|f(x)-f(y)|$.
For each $\nu$ we decompose $f$ into a good part and a bad part.
Let $A_\nu$ be the set of all roots $\alpha$ 
for which $\dist(\alpha, I_\nu) < \delta$.
$$
\begin{array}{ll}
g_\nu := \sum_{\alpha\not\in A_\nu} \log |x-\alpha|, &
b_\nu := \sum_{\alpha\in A_\nu} \log |x-\alpha|.
\end{array}
$$
We now define bad subsets $E_{j,\eps}$:
$$
\begin{array}{ll}
E_{1, \eps} := \{x\in I:\ |f'(x)|> 1/(\eps\delta)\}\ ,&
E_{2, \eps} := \{x\in I:\ \exists \alpha, |x-\alpha| < \eps\delta\}\ ,\\
E_{3, \eps} := \cup\{I_\nu: |A_\nu|>1/\eps\}\ ,&
E_{4, \eps} := \cup\{I_\nu : \int_{I_\nu} |g_\nu''(x)| \dx > 
1/(\eps\delta)\}\ ,\\ 
E_{5,\eps} := E_{1, \eps}\cup E_{2,\eps}\cup E_{3,\eps}\cup E_{4,\eps}\ ,&
E_{6,\eps} := \cup\{I_\nu: |I_\nu\cap E_{5,\eps}|/|I_\nu|>1/2\}\ ,\\
E_{\eps}   := E_{5,\eps}\cup E_{6,\eps}\ .&
\end{array}
$$
\begin{lemma}
\label{lem:b-b}
Let $x\in I_\nu \setminus (E_{2, \eps}\cup E_{3, \eps})$.
Then,
 $$\forall y\in I_\nu,\ b_\nu(x)-b_\nu(y) \geq 
-\frac{1}{\eps}\log\frac{3}{\eps}\ .$$
\end{lemma}
\begin{proof}
For all $y\in I_\nu$
$$b_\nu(y) = \sum_{\alpha\in A_\nu} \log |y-\alpha| 
\leq (\log3\delta)/\eps\ ,$$
while $b_\nu(x) \geq (\log(\eps\delta))/\eps$.
\end{proof}
\begin{lemma}
\label{lem:b'}
Let $x\in I_\nu\setminus (E_{2,\eps}\cup E_{3,\eps})$. Then,
$|b_\nu'(x)| \leq 1/(\eps^2\delta)$.
\end{lemma}
\begin{proof}
Since $x\not\in E_{2,\eps}\cup E_{3, \eps}$.
$|b_\nu'(x)|\leq \sum_{\alpha\in A_\nu} 
\frac{1}{|x-\alpha|}\leq (1/\eps)\cdot 1/(\eps\delta)\ . $
\end{proof}
\begin{lemma}
\label{lem:g'}
Let $x\in I_\nu\setminus (E_{1, \eps}\cup E_{2, \eps}\cup E_{3, \eps})$.
Then, 
$$|g_\nu'(x)| \leq 2/(\eps^2\delta)\ .$$
\end{lemma}
\begin{proof}
Since $x\not\in E_{1,\eps}$, $|f'(x)|\leq 1/(\eps\delta)$.
By Lemma~\ref{lem:b'}
$$|b_\nu'(x)|\leq 1/(\eps^2\delta)\ . $$
It only remains to observe that
$|g_\nu'| \leq |f'|+|b_\nu'|$.
\end{proof}
\begin{lemma}
\label{lem:g'-on-good}
Suppose $I_\nu \not\subseteq E_{5,\eps}$.
Then $\max_{I_\nu} |g_\nu'| \leq 3/(\eps^2\delta)$.
\end{lemma}
\begin{proof} 
By Lemma~\ref{lem:g'} $\exists x_\nu\in I_\nu$ such that 
$$|g_\nu'(x_\nu)| \leq 2/(\eps^2\delta)\ .$$
Also, from the definition of $E_{4, \eps}$ and
the fundamental theorem of calculus
 $$\forall x\in I_\nu\ |g_\nu'(x) - g_\nu'(x_\nu)| \leq 1/(\eps\delta)\ .$$
Together we obtain
$\forall x\in I_\nu\ |g_\nu'(x)| \leq 3/(\eps^2\delta)$.
\end{proof}
\begin{lemma}
\label{lem:g-g}
Suppose $I_\nu\not\subseteq E_{5,\eps}$.
Then
  $$\forall x, y\in I_\nu\ |g_\nu(x) - g_\nu(y)| \leq 6/\eps^2\ .$$
\end{lemma}
\begin{proof}
The proof is an immediate corollary of Lemma~\ref{lem:g'-on-good}.
\end{proof}
\begin{lemma}
\label{lem:f-f}
Let $x\in I_\nu\setminus E_{5, \eps}$.
  $$\forall y\in I_\nu,\ f(x) - f(y) \geq -9/\eps^2\ .$$
\end{lemma}
\begin{proof}
$f(x)-f(y) = (g_\nu(x)-g_\nu(y))+(b_\nu(x)-b_\nu(y))$.
It only remains to combine Lemmas~\ref{lem:b-b} and~\ref{lem:g-g}.
\end{proof}
\begin{lemma}
Let $x\in I_\nu\setminus E_\eps$.
Then
$$\me^{-9/\eps^2}/4
\leq \frac{\me^{f(x)}}{\Av_{I_{\nu}}\me^f} \leq 4\me^{9/\eps^2}$$
\end{lemma}
\begin{proof}
On the one hand, Lemma~\ref{lem:f-f} gives
\begin{multline*}
\frac{\me^{f(x)}}{(\int_{I_\nu}\me^f\dx)/|I_\nu|}
\leq \frac{\me^{f(x)}}{(\int_{I_\nu\setminus E_\eps}\me^f\dx)/|I_\nu|}
= \\ \frac{\me^{f(x)-\max_{I_{\nu}f}}}{(\int_{I_\nu\setminus E_\eps}
\me^{f(x)-\max_{I_\nu} f}\dx)/|I_\nu|}\leq
\frac{\me^{f(x)-\max_{I_\nu}f}}
{(|I_\nu\setminus E_\eps|/|I_\nu|)\me^{-9/\eps^2}/2}\leq
\frac{1}{\me^{-9/\eps^2}/4}=4\me^{9/\eps^2}\ .
\end{multline*}
On the other hand, Lemma~\ref{lem:f-f} also gives
$$\frac{\me^{f(x)}}{(\int_{I_\nu}\me^f\dx)/|I_\nu|} \geq
\frac{\me^{f(x)}}{\me^{\max_{I_\nu} f}}
= \me^{f(x)-\max_{I_\nu}f}\geq
\me^{-9/\eps^2}\ .$$
\end{proof}
We now turn to estimating the size of the bad subset $E_\eps$.
\begin{lemma}
\label{lem:hilbert-transform}
  $|E_{1,\eps}|< C_9\eps\mu\delta$
\end{lemma}
\begin{proof}
 This follows from the properties of the Hilbert Transform.
 We imitate the proof of Lemma~5.4 in~\cite{DF} with a little more details.

We recall the definition and some basic properties of the Hilbert
Transform.
 Let $u\in L^2(\RR)$. Let $\sgn$ be the sign function on $\RR$.
 Let $\cF$ be the Fourier Transform on $L^2(\RR)$.
 Define the Hilbert Transform $\cH u$ by
 $$ \cF(\cH u) = \frac{i}{2}\sgn\cdot \cF(u) \ .$$
 From this definition it is clear that $\cH$ is a bounded
 operator on $L^2(\RR)$.
Observe that
$$f'(x) =\sum_\alpha \Re\left(\frac{1}{x-\alpha}\right)\ .$$
  We may assume $\forall\alpha,\ \Im\alpha \leq0$.
  Consider first the case where $\forall\alpha,\ \alpha\not\in\RR$.
  Let $q_\alpha(x) = -\Im (1/(x-\alpha))$, and $q=\sum_\alpha q_\alpha$.
 Then, $q\in L^1(\RR)\cap L^2(\RR)$ and by Theorem~3
 in~III.2.3 of~\cite{stein-singular} $\cH q = f'$.
 From the fact that $\sgn'=2\delta_0$ and
 by basic properties of the Fourier Transform one sees
 that if $u\in L^2(\RR)$ has a compact support
 and $x\not\in\supp u$, then
 $$(\cH u)(x) = \int_\RR \frac{u(y)}{x-y} \dy \ $$
  (See also exc.~1.9 in~\cite{microlocal}
 and Theorem~5 in~III.3.3 of~\cite{stein-singular}).
 We have verified that the conditions of Theorem~3 in~I.5 of~\cite{stein-harmonic}
 are fulfilled for the Hilbert Transform.
 We conclude that the Hilbert Transform is of weak type~(1,1) and we get
   \begin{equation}
   \label{weak-type-1-1}
   |\{|f'| > 1/(\eps\delta)\}| \leq C_{10}\eps\delta\|q\|_{1}
   \leq C_{11}\eps\mu\delta \ .
   \end{equation}

Finally, we move to the case where $\exists\alpha\in \RR$. Define
$g_t(x) : = f'(x-it)$. A small calculation shows that
$g_t\to f'$ in measure as $t\to 0$. Since we can apply the
considerations above to $g_t$ we conclude
that the assertion in the lemma is true with $C_9=2C_{11}$.
\end{proof}
\begin{lemma}
$$|E_{2,\eps}| \leq C_{12}\eps\mu\delta\ .$$
\end{lemma}
\begin{proof}
Proof is obvious.
\end{proof}
\begin{lemma}
$$|E_{3,\eps}|\leq C_{13}\eps\mu\delta\ .$$
\end{lemma}
\begin{proof}
This is an immediate corollary of 
Lemma~\ref{lem:num-roots}.
\end{proof}
\begin{lemma}
   $|E_{4,\eps}| \leq C_{14}\eps\mu\delta$.
\end{lemma}
\begin{proof}
We observe that
$$g_{\nu}''(x) = -\sum_{\alpha\not\in A_\nu} 
\Re\left(\frac{1}{(x-\alpha)^2}\right)\ .$$
Hence,
\begin{multline*}
  \sum_{\nu} \int_{I_\nu} |g_\nu''(x)| \dx \leq
   \sum_{\nu}\sum_{\alpha\not\in A_\nu} 
\int_{I_\nu}  \frac{1}{|x-\alpha|^2}\dx
   =\\ \sum_\alpha \sum_{\nu, \alpha\not\in A_\nu} 
   \int_{I_\nu} \frac{1}{|x-\alpha|^2}\dx \leq
   \sum_\alpha \int_{|x-\alpha|>\delta} \frac{1}{|x-\alpha|^2}\dx
   \leq \mu/\delta.
\end{multline*}

   On the other hand
     $$\sum_{\nu} \int_{I_\nu} |g_\nu''(x)| \dx
     \geq \sum_{\nu, I_\nu\subseteq E_{4, \eps}} \int_{I_\nu} |g_\nu''(x)| \dx
     \geq \#\{\nu: I_\nu\subseteq E_{4, \eps}\} 1/(\delta\eps) \ .$$
   Together, we get that
   $  \#\{\nu: I_\nu\subseteq E_{4, \eps}\} \leq \eps\mu $.
   Hence, $|E_{\eps, 4}| \leq C_{14}\eps\mu\delta$.
\end{proof}
\begin{lemma}
$|E_{6, \eps}| \leq C_{15}\eps\mu\delta$
\end{lemma}
\begin{proof}
Let $N$ be the number of intervals $I_\nu$
for which $|I_\nu\cap E_{5,\eps}|/|I_\nu|>1/2$. We have
$$C_{16}\eps\mu\delta\leq |E_{5,\eps}| = 
\sum_\nu |I_\nu\cap E_{5,\eps}| \geq
N|I_\nu|/2\geq N\delta/2\ .$$
Hence, $N\leq 2C_{16}\eps\mu$. 
It follows that $|E_{6,\eps}|\leq 4C_{16}\eps\mu\delta$.
\end{proof}

This completes the proof of Proposition~\ref{prop:wildintervals}.

\section{Holomorphic Functions in Small Scales - dimension $n>1$}
\label{sec:hol-dim-n}
In this section we
prove an  analog of Proposition~\ref{prop:wildintervals} in
dimension $n>1$. We adjust the proof of Proposition~5.11
in~\cite{DF}.

Let $F$ be a holomorphic function defined in the polydisk 
$B_3^n=B_3\times\ldots\times
B_3\subseteq \CC^n$.
Let $I=[-1, 1]\subseteq \RR$ and $Q = I^n$.
Assume $F$ satisfies
\begin{equation}
\label{growth-dim-n}
 \sup_{B_2^n} |F| \leq |F(0)|\me^{C_1\mu}\ .
\end{equation}
  Given $0<\delta<C_2/\mu$, decompose $Q$ into
  subboxes $Q_\nu$ in the following way:
First, we define $n$ decompositions of $I$ into intervals 
$\{I^{(k)}_l\}$ where $\delta<|I^{(k)}_l|<2\delta$ $1\leq k\leq n$ and $1\leq l\leq N(k)$.
Given a multi-index $\nu=(\nu_1, \ldots, \nu_n)$, $1\leq \nu_k\leq N(k)$,
we set 
$Q_\nu = I^{(1)}_{\nu_1}\times\ldots\times I^{(n)}_{\nu_n}$. 
Given $x\in Q$, we denote by $Q_x$ the subbox which contains $x$.
We prove
\begin{proposition}
\label{prop:F-AvF-technic}
     Let $F$ satisfy~(\ref{growth-dim-n}) and $F\geq 0$ on $Q$.
     Assume that $F\equiv 1$ on each of the hyperplanes $z_i=0$.
   For all $\eps>0$ there exists a subset 
   $E_\eps\subseteq Q$ of measure
  $|E_\eps|\leq C_3\eps\mu\delta$ such that
\begin{equation}
\label{approximate-average}
   \frac{1}{C_4(\eps)}\leq \frac{F(x)}{\Av_{Q_x} F} \leq C_4(\eps) \quad 
\forall x\in Q\setminus E_\eps ,
\end{equation}
with $C_4(\eps)\to\infty$ as $\eps\to 0$.
\end{proposition}
\begin{proof}
For $n=1$, the proposition reduces to Proposition~\ref{prop:wildintervals}.
For $n>1$,
let $E_\eps$ be the subset of
all $x\in Q$  for which the inequalities in~(\ref{approximate-average})
are not true with $C_4(\eps) = \me^{11n/\eps^2}$.
Given $z'=(z_1,\ldots z_{n-1})\in B_3^{n-1}$ 
we define 
$$F_{z'}(z) = F(z', z)\ .$$
$F_{z'}$ has the following properties:
\begin{itemize}
\item
$F_{z'}$ is defined in $B_3$.
\item
If $z'\in B_2^{n-1}$, then $\sup_{B_2}|F_{z'}| \leq \me^{C_1\mu}$. 
\item Let $Q'=I^{n-1}$. If $x'\in Q'$ then
$F_{x'}\geq 0$ on the interval $[-1,1]$.
\item If $x'\in Q'$ then
$F_{x'}(0)= 1$. 
\end{itemize}

We have checked that for $x'\in Q'$, 
$F_{x'}$ satisfies the conditions in Proposition~\ref{prop:wildintervals}.
Let $E_\eps(F_{x'})\subseteq I$ be the corresponding bad subset.
We let $E^{x'}_\eps := \{x'\}\times E_\eps(F_{x'})$.
Let $E^{(n)}_\eps := \cup_{x'\in Q'} E^{x'}_\eps$.
$E^{(n)}_\eps$ might not be measurable, but it intersects
every line parallel to the $x_n$-axis in a measurable set.

Given $1\leq l\leq N(n)$, $z'\in B_3^{n-1}$
define, 
  $$G_l(z') := \Av_{I_l^{(n)}} F_{z'}\ .$$
It is easy to check that
\begin{itemize}
\item
$G_l$ is defined in $B_3^{n-1}$.
\item
$\sup_{B_2^{n-1}}|G_l| \leq \me^{C_1\mu}$. 
\item
$G_l(x')\geq 0$ for $x'\in Q'$.
\item
$G_l(z')= 1$ whenever one of the coordinates
$z_i=0$.
\end{itemize}

Thus, by the induction hypothesis applied to $G_{l}$ and
the decomposition $Q_{\nu'}'$ we get a corresponding bad subset 
$E_\eps(G_{l})$.
We set $E_\eps^l := E_\eps(G_l)\times I^{(n)}_{l}$.
Let $E_\eps' := \cup_{1\leq l\leq N(n)} E_\eps^l$.

\begin{claim}
\label{claim:bad-inclusion}
$E_\eps\subseteq E_\eps'\cup E^{(n)}_\eps$.
\end{claim}

\begin{proof}
Let $x\in Q_\nu\setminus (E_\eps'\cup E_\eps^{(n)})$.
Since $x'\not\in E_\eps(G_{\nu_n})$ we have
\begin{equation}
\label{G_l}
\me^{-11(n-1)/\eps^2}\leq
\frac{G_{\nu_n}(x')}{\Av_{Q'_{\nu'}} G_{\nu_n}} \leq \me^{11(n-1)/\eps^2}\ .
\end{equation}

Since $x_n\not\in E_\eps(F_{x'})$ we have
\begin{equation}
\label{F_x}
\me^{-11/\eps^2}\leq \frac{F_{x'}(x_n)}{\Av_{I^{(n)}_{\nu_n}} F_{x'}}
\leq \me^{11/\eps^2}\ .
\end{equation}
Now recall that 
$$F_{x'}(x_n) = F(x),\ \Av_{I^{(n)}_{\nu_n}} F_{x'} = G_{\nu_n}(x'), $$
and observe that ${\Av_{Q'_{\nu'}} G_{\nu_n}} = \Av_{Q_\nu} F$.
To complete the proof of Claim~\ref{claim:bad-inclusion}
we multiply~(\ref{G_l}) by~(\ref{F_x}).
\end{proof}

It only remains to check that the size of $E_\eps$ is not too big:
By Claim~\ref{claim:bad-inclusion} we know that
$E_\eps\setminus E_\eps'$ is a measurable set all of whose
intersections with lines parallel to the $x_n$-axis are measurable
sets of sizes $\leq C_5\eps\mu\delta$.
By Fubini's Theorem, we get $|E_\eps\setminus E_\eps'| \leq C_6\eps\mu\delta$.
$|E_\eps'|\leq \sum_l 2|E_\eps(G_l)|\delta \leq C_7N(n)\eps\delta^2\mu$. 
But $N(n)\leq 2/\delta$.
This completes the proof of Proposition~\ref{prop:F-AvF-technic},
since $|E_\eps|\leq |E_\eps\setminus E_\eps'|+|E_\eps'|$.
\end{proof}

We now remove the technical assumption in proposition~\ref{prop:F-AvF-technic}.
The main proposition of this section is
\begin{proposition}
\label{prop:control-quotient}
  Let $F$ satisfy~(\ref{growth-dim-n}) and $F\geq 0$ on $Q$.
  There exists a cube $R\subseteq Q$ independent of $F$
  with the following property:
  Suppose $\mu\delta <C_8$. We decompose $R$ into boxes $R_\nu$
  of sides $\delta<l_\nu^{(k)}<2\delta$. Then, there exists a subset
  $E_\eps\subseteq R$ of measure $|E_\eps|\leq C_{9}\eps\mu\delta$
  such that
\begin{equation}
\label{approximate-average-dim-n}
     \frac{1}{C_{10}(\eps)}\leq \frac{F(x)}{\Av_{R_x} F} \leq C_{10}(\eps) 
  \quad \forall x\in R\setminus E_\eps,
\end{equation}
with $C_{10}(\eps)\to\infty$ as $\eps\to 0$.
\end{proposition}

\begin{proof}
We construct $R$ in the same way as in~\cite{DF}:
\begin{lemma}[{{\cite[Lemma 5.10]{DF}}}]
   There exists a map $W:\RR^n\to \RR^n$ which
extends to a map $\hat{W}:\CC^n\to \CC^n$.
   with the following properties:
    \begin{enumerate}
      \item $W$ is a polynomial map.
      \item $\hat{W}(B_3^n)\subseteq B_3^n$.
      \item $\hat{W}(B_2^n)\subseteq B_2^n$.
      \item $W(Q)\subseteq Q$.
      \item $W$ maps the hyperplanes $x_i =0$ to $0$.
      \item $W$ is a local diffeomorphism outside
             the hyperplanes $x_i=0$.
    \end{enumerate}
\end{lemma}

Let $U\subseteq Q$ be an open set which is mapped
diffeo\-morphic\-ally onto $W(U)$ and has a positive distance from
any hyperplane $x_i=0$. We let $R\subseteq W(U)$ be any
cube with sides parallel to the sides of $Q$.

Let us now describe the bad subset $E_\eps$. We begin with
\begin{lemma}
\label{lem:sub-div} There exists a finite number of subdivisions
$\cD_i$ of $Q$ into boxes $Q_{\nu, i}$ of sides 
$\delta<|I^{(k)}_{l,i}|<2\delta$
such that every set of diameter $<\delta/2$ is contained in 
a box $Q_{\nu,i}$ for some $\nu$ and $i$.
\end{lemma}
The function $\tld{F}=F\circ \hat{W}/F(0)$ satisfies the conditions
of  Proposition~\ref{prop:F-AvF-technic}. So, given any of the
subdivisions $\cD_i$ of Lemma~\ref{lem:sub-div}
 we can find an
 exceptional set $\tld{E}_{\eps, i} \subseteq Q$ corresponding to $\tld{F}$.
 Let $\tld{E}_\eps= \cup_i \tld{E}_{\eps, i}$.
Set $E_\eps^0 = W(\tld{E}_\eps\cap U)\cap R$.

Call $R_\alpha$ \emph{$\eps$-bad} if $|E_\eps^0\cap R_\alpha|/|R_\alpha|>1/2$.
Let $B_\eps$ be the union of all $\eps$-bad $R_\alpha$'s.
Finally, set $E_\eps = E_\eps^0\cup B_\eps$.
%
We now estimate the size of $E_\eps$.
\begin{lemma}
  $|E_\eps| \leq C_{11}\eps\mu\delta$.
\end{lemma}
\begin{proof}[Proof of Lemma]
Since the Jacobian of the map $W$ is bounded on $U$,
and $|\tld{E}_\eps|\leq C_{12}\eps\mu\delta$ we conclude
that $|E_\eps^0|\leq C_{13}\eps\mu\delta$.
We estimate $|B_\eps|$:
\begin{equation*}
C_{13}\eps\mu\delta\geq |E_\eps^0| \geq \sum_{\mbox{bad }
R_{\alpha}\mbox{'s}} |E_\eps^0\cap R_{\alpha}| \geq
(1/2)\#(\mbox{bad } R_{\alpha}\mbox{'s}) |R_{\alpha}| \geq C_{14}|B_\eps|\ .
\end{equation*} 
We got $|E_\eps|\leq |E_\eps^0|+|B_\eps|\leq C_{11}\eps\mu\delta$.
\end{proof}

The last step is to check that~(\ref{approximate-average-dim-n})
is true:
 Let $R_\alpha$ be a subbox of $R$ with sides 
$C_{15}\delta\leq l^{(k)}_\alpha\leq 2C_{15}\delta$, where $C_{15}$ 
is small enough. Look at $\tld{R}_\alpha =W^{-1}(R_\alpha)$.
 Since $W^{-1}$ has a bounded Jacobian on $R$,
 $\tld{R}_\alpha$ is a set of diameter $<\delta/2$.
 Let $\cD$ be one of the subdivisions of 
$Q$ from Lemma~\ref{lem:sub-div} whose one of its boxes $Q_\nu$ contains 
$\tld{R}_\alpha$.

It follows from Proposition~\ref{prop:F-AvF-technic} that
 $\tld{F}(y_1)/\tld{F}(y_2) \leq {C_4(\eps)}^2$ $\forall y_1, y_2\in 
Q_\nu \setminus \tld{E}_\eps.$
 Hence, if we let $x_0\in R_\alpha\setminus E_\eps$ and $y_0=W^{-1}(x_0)$,
 then $y_0\in\tld{R}_\alpha\setminus\tld{E}_\eps$ and we obtain
\begin{multline}
\label{av>pt}
\Av_{R_\alpha} F = \frac{1}{|R_\alpha|}\int_{R_\alpha} F(x) \dx \geq
\frac{1}{|R_\alpha|}\int_{R_\alpha\setminus E_\eps} F(x) \dx = \\
\frac{1}{|R_\alpha|}\int_{\tld{R}_\alpha\setminus \tld{E}_\eps} 
\tld{F}(y)|J_W| \dy \geq
\frac{1}{C_4(\eps)^2|R_\alpha|}\int_{\tld{R}_\alpha\setminus \tld{E}_\eps} 
\tld{F}(y_0)|J_W|\dy
= \\
\frac{|R_\alpha\setminus E_\eps|}{C_4(\eps)^2|R_\alpha|} F(x_0) \geq 
F(x_0)/(2C_4(\eps)^2).
\end{multline}

On the other hand,
\begin{multline}
\label{av<pt}
 \Av_{R_\alpha} F= \frac{1}{|R_\alpha|}\int_{R_\alpha} F(x) \dx =
  \frac{1}{|R_\alpha|}\int_{\tld{R}_\alpha} \tld{F}(y) |J_W|\dy \leq \\
   \frac{1}{|R_\alpha|}\int_{Q_\nu} \tld{F}(y) |J_W|\dy \leq \\
   \frac{C_{16}|Q_\nu|}{|R_\alpha|}\frac{1}{|Q_\nu|}\int_{Q_\nu} \tld{F}(y) \dy
   \leq C_{17}\Av_{Q_\nu} \tld{F} \leq C_{17}C_4(\eps)\tld{F}(y_0) 
= C_{18}(\eps) F(x_0).
\end{multline}

Inequalities~\eqref{av>pt} and~\eqref{av<pt} complete the proof of
Proposition~\ref{prop:control-quotient}.
\end{proof}
\section{Eigenfunctions in Small Scales on Real Analytic Manifolds}
\label{sec:eigen-small-scale}
Let $\phi_\mu$ be an eigenfunction.
Let $V$ be a small open set in which the metric $g$ can
be developed in power series.
We identify $V$ with
a ball $B(0, \rho_0) \subseteq \RR^n$.
We prove
\begin{proposition}
  \label{prop:eigen-small-scale}
  There exists a cube $R\subseteq V$ with the following property:
  Suppose $\mu\delta<C_1$.
  We subdivide $R$ into boxes $R_\nu$ of sides 
$\delta <l^{(k)}_{\nu}< 2\delta$. Then for all $\eps>0$
  there exists a subset $E_\eps\subseteq R$ of 
  measure $|E_\eps|\leq C_2\eps\mu\delta$
  such that
  $$\frac{1}{C_3(\eps)}\leq \frac{\phi_\mu(x)^2}{\Av_{R_x} \phi_\mu^2} 
   \leq C_3(\eps) \quad
  \forall x\in R\setminus E_\eps, $$
  with $C_3(\eps)\to\infty$ as $\eps\to 0$.
\end{proposition}

\begin{proof}
We consider an analytic continuation of $\phi_\mu$.
In order to avoid confusion we denote by $B^\RR$, $B^\CC$ balls
in $\RR^n, \CC^n$ respectively.
\begin{lemma}[{\cite[Lemma 7.3]{DF}}]
\label{lem:analytic-continuation}
$\phi_\mu|_{B^\RR(0, \rho_0)}$ has an analytic continuation $F$
defined on $B^\CC(0 ,\rho_1)$
for some $\rho_1 <\rho_0$.
Moreover, the function $F$ satisfies
$$\sup_{B^\CC(0, \rho_1)} |F| 
\leq \me^{C_4\mu}\sup_{B^\RR(0, \rho_0)} |\phi_\mu|\ .$$
\end{lemma}
The crucial point is that the domain to which the 
function $\phi_\mu$ can be continued is independent of $\mu$.

Let $\rho_2 = \rho_1/C_5$ with $C_5$ large so that
the polydisk $B_{2\rho_2}^n\subseteq B^\CC(0, \rho_1)$.
We now recall the Donnelly-Fefferman Growth Bound
\begin{theorem}
\label{thm:df-growth}
$$\sup_{B^\RR(0, \rho_0)} |\phi_\mu| \leq \me^{C_6(\rho_0/\rho_2)\mu} 
\sup_{B(0, \rho_2)} \phi_\mu\ .$$
\end{theorem}
Lemma~\ref{lem:analytic-continuation} and Theorem~\ref{thm:df-growth}
give
$$\sup_{B_{2\rho_2}^n} |F| \leq 
\me^{C_7\mu} \sup_{B(0, \rho_2)} \phi_\mu\ .$$
Now, shift the coordinate system to be centered on
the point $x\in B(0, \rho_2)$ for which 
$\phi_\mu(x) = \sup_{B(0, \rho_2)} \phi_\mu$.
We get that 
$$\sup B_{\rho_2}^n |F| \leq \me^{C_7\mu} |F(0)|\ .$$

Hence, we can conclude the proof by applying 
Proposition~\ref{prop:control-quotient} to $F^2$.
\end{proof}

We need a slightly different version of this proposition.
We say that $R_\alpha$ touches $R_\beta$ if they have at least one vertex
in common. Each box $R_\alpha$ touches at most $3^n$ boxes.
Let us denote by $R_\alpha^*$ the union of the box $R_\alpha$ with
all boxes which touch $R_\alpha$.
There exist $3^n$ subdivisions $\cD_i$ of $R$
such that each box of $\cD_i$ is equal to $R_\alpha^*$ for some $\alpha$.
Let $E_\eps=\cup_i E_{\eps, i}$ where  $E_{\eps, i}$ is the bad subset 
corresponding to the
subdivision $\cD_i$ according to Proposition~\ref{prop:eigen-small-scale}.
$|E_\eps|\leq C_{8}\eps\mu\delta$. 
These considerations prove the following version of
Proposition~\ref{prop:eigen-small-scale}.
\begin{proposition}
  \label{prop:eigen-small-scale-v2}
  There exists a cube $R\subseteq V$ with the following property:
  Suppose $\mu\delta<C_{9}$.
  We subdivide $R$ into boxes $R_\nu$ of sides 
$\delta <l_\nu^{(k)} <2\delta$.
  There exists a subset 
$E_{\eps}\subseteq R$ of measure $|E_{\eps}|\leq C_{8}\eps\mu\delta$
  such that
  $$\frac{1}{C_{10}(\eps)}\leq \frac{\phi_\mu(x)^2}{\Av_{R_x^*} \phi_\mu^2}
  \leq C_{10}(\eps)
  \quad\forall x\in R\setminus E, $$
with $C_{10}(\eps)\to\infty$ as $\eps\to 0$.
\end{proposition}

\section{Good Boxes - Bad Boxes}
\label{sec:good-bad}
Let $F$ be a nonnegative function defined on a cube $R$.
Let $\Dl$ be a subdivision of $R$.
We divide the boxes $R_\nu$ into
\emph{good} and \emph{bad}. 
We show that in a vicinity of a good
box we have a bounded $L^2$-growth, 
and that the geometry is under control.
We show that the proportion
of bad boxes is small.

We always assume that the sides of all boxes of a subdivision
are of comparable sizes. Moreover, we assume that any
two boxes $R_{\nu^1}, R_{\nu^2}$
satisfy 
$$\frac{\max\side(R_{\nu^1})}{\min\side(R_{\nu^2})}\leq 5\ .$$
We recall that $R_\nu^*$ denotes the union of $R_\nu$ with
its $3^n-1$ neighbors.

\begin{definition}
\label{def:approximate-average}
Let $E\subseteq R$ be such that
$$\frac{1}{A}\leq\frac{F(x)}{\Av_{R_\nu^*} F} \leq A
\quad\forall\nu\ \forall x\in R_\nu\setminus E\ .$$
We say that \emph{$(F, \Dl, E, A)$ is true}.
\end{definition}

\begin{definition}
\label{def:goodness}
Given $E\subseteq R$,
we say that $R_\nu$ is \emph{$E$-good} if
$|E\cap R_\nu|/|R_\nu| < \omega_n 10^{-2n}$,
where $\omega_n$ is the volume of the unit ball in $\RR^n$.
Otherwise $R_\nu$ is called \emph{$E$-bad}.
\end{definition}

The next lemma shows that in the vicinity of any good box 
we have bounded growth.
$2Q$ denotes a box concentric with $Q$, whose
sides are parallel to the sides of $Q$ and twice as large.
\begin{lemma}
\label{lem:growth-on-good-balls}
Suppose that $(F, \Dl, E, A)$ is true.
Let $R_\nu$ be $E$-good. Let $B\subseteq R_\nu$ be a ball such that
$2B\subseteq R_\nu^*$ and whose radius $r\geq\side(R_\nu)/20$.
Then,
 \begin{equation*}
   \int_{B} F \dx \geq C_1 A^{-1}\int_{2B} 
F \dx
 \end{equation*}
\end{lemma}
\begin{proof}
\begin{multline*}
  \int_{B} F(x) \dx \geq
  \int_{B\setminus E} F(x) \dx \geq
  A^{-1}\int_{B\setminus E} \Av_{R_\nu^*} F \dx  = \\
  A^{-1}\frac{|B\setminus E|}{|R_\nu^*|} \int_{R_\nu^*} 
F \dx \geq
 A^{-1}\frac{|B|-|E\cap R_\nu|}{|R_\nu^*|}\int_{2B} F \dx \geq
 \\  \omega_n(20^{-n}-10^{-2n})A^{-1} \frac{|R_\nu|}{|R_\nu^*|}
\int_{2B} F \dx \geq
C_1 A^{-1} \int_{2B} F \dx\ .
\end{multline*}
\end{proof}

The next proposition shows that the geometry
in good cubes is controlled.
\begin{proposition}
\label{prop:geometry-control}
Suppose $(\phi_\mu^2, \Dl, E, A)$ is true.
Let $R_\nu$ be $E$-good.
Let $B=B(o, 2r)\subseteq R_\nu$ be a ball
of radius $2r$, with $r>\side(R_\nu)/20$. 
Let $B^+ = B\cap\{\phi_\mu>0\}$. Similarly, define $B^-$.
Suppose $\phi_\mu(o)=0$.
Then $$\frac{1}{C_2(\log A)^{n-1}} 
\leq \frac{\Vol(B^+)}{\Vol(B^-)} \leq C_2(\log A)^{n-1}\ .$$
\end{proposition}
\begin{proof}
After rescaling the ball $B$ to the unit ball,
$\phi_\mu$ becomes a solution $\vphi$ of an elliptic equation
\begin{equation}
\label{L-sol}
 -\partial_i(a^{ij}\partial_j\vphi)-(2r\mu)^2q\vphi=0\ .
\end{equation}
It's important to observe that the coefficients are bounded
independently of $\mu$, and the zero order coefficient is small.
Thus, by Lemma~\ref{lem:growth-on-good-balls} and by elliptic regularity
\begin{equation}
\label{max-growth-bound}
  \sup_{B_{3/4}}{|\vphi|} \leq C_3\|\vphi\|_{L^2(B_1)}
 \leq C_4 A^{1/2}\|\vphi\|_{L^2(B_{1/2})} \leq 
C_5 A^{1/2}\sup_{B_{1/2}}{|\vphi|}\ .
\end{equation}
Recall now
\begin{theorem}[{\cite[Theorem 4.7]{Man-LAIRND}}]
Let $\vphi$ satisfy equation~(\ref{L-sol}) in the unit ball $B_1$.
Suppose $\vphi(0)=0$ 
and satisfies~(\ref{max-growth-bound}).
Then
\begin{equation}
\label{growth-vs-positivity}
  \frac{\Vol(B_1\cap\{\vphi>0\})}{\Vol(B_1)} \geq \frac{C_6}{(\log A)^{n-1}} \ .
\end{equation}
\end{theorem}
By symmetry, we have a lower bound also on 
$\Vol(B_1\cap\{\vphi<0\})/\Vol(B_1)$. Thus, we get upper and lower bounds
on the ratio between the volumes of the positivity
and the negativity sets of $\vphi$.
\end{proof}

The last lemma in this section shows that the proportion
of bad cubes is small.
\begin{lemma}
\label{lem:bad-proportion}
$$\frac{\#(E\mbox{-bad boxes})}{\#(\mbox{all boxes})} \leq 
C_{7}|E|\ .$$
\end{lemma}
\begin{proof}
\begin{equation*}
|E| \geq \sum_{\mbox{bad } R_\nu\mbox{'s}} |E\cap R_\nu|
\geq \omega_n {10}^{-n} \#(\mbox{bad } R_\nu\mbox{'s}) |R_\nu| \geq
C_{8}\frac{\#(\mbox{bad } R_\nu\mbox{'s})}{\#(\mbox{all boxes})}\ .
\end{equation*}
\end{proof}

\section{Refinements of subdivisions}
\label{sec:refinements}
In this section we analyze what happens
when we pass from a fine subdivision $\Dl_1$ to a 
subdivision $\Dl$ of whose $\Dl_1$ is a refinement.
We show, roughly, that if a box of $\Dl$ is composed of smaller good boxes
then it is also good. The results in this section are applied in the
the proof of the lower bound in Theorem~\ref{thm:nodal-tube}.

Let $\Dl^1$ be a subdivision of a cube $R$ obtained by a refinement
of $\Dl$. If the sides of every box in $\Dl$ are partitioned
into $\leq M$ intervals, we write $[\Dl:\Dl^1]\leq M$.
Let $\Dl^2$ be the subdivision of $R$ which is formed by 
taking the centers of the boxes in $\Dl^1$.

Throughout this section $F$ is a nonnegative function defined on $R$.
In the next proposition we use the terminology from 
Definitions~\ref{def:approximate-average} and~\ref{def:goodness}.
\begin{proposition}
\label{prop:refinement}
Suppose that $(F, \Dl^1, E^1, A)$ and 
$(F, \Dl^2, E^2, A)$ are true.
Let $B$ be the union of all boxes $R_\nu$ of $\Dl$
for which $R_\nu^*$ contains an $E^1$-bad box of $\Dl^1$ or
an $E^2$-bad box of $\Dl^2$.
Assume $[\Dl:\Dl^1] \leq M$. Then,
$(F, \Dl, E^1\cup E^2\cup B, C_1 A^{8M+1})$ is true.
\end{proposition}
\begin{proof}
Let us denote the boxes of $\Dl$ by $R_\nu$,
and the boxes of $\Dl^1$ by $R_\alpha$. Let $E=E^1\cup E^2\cup B$.
\begin{lemma}
\label{lem:orbits}
For all $x\in R_\nu\setminus E$, $y\in R_\nu^*\setminus E$ 
there exists a sequence of $4M$ points 
$x=x_1, \ldots, x_{4M}= y$ in $R_\nu^*\setminus E$
such that any consecutive pair $x_k, x_{k+1}$ is
contained in a box of $\Dl^1$ or is contained in a box
of $\Dl^2$.
\end{lemma}
\begin{proof}[Proof of Lemma]
The only point to observe is that
if $R_\nu\not\subseteq E$ then
all boxes $R_\alpha$ 
contained in $R_\nu^*$ are $E^1$-good. So 
$|E^1\cap R_\alpha|/|R_\alpha| <\omega_n 10^{-2n}<2^{-n}/2$.
Similarly for $E^2$.
\end{proof}
Let $x\in R_\nu\setminus E$.
For any $y\in R_\nu^*\setminus E$, let $x_1, \ldots, x_{4M}$
be a sequence of points as in Lemma~\ref{lem:orbits}.
Since $(F, \Dl^1, E^1, A)$ and
$(F, \Dl^2, E^2, A)$ are true we have 
$F(x_k)/F(x_{k+1}) \leq A^2$, and we get
$F(x)/F(y)\leq A^{8M}$.

Now,
\begin{multline*}
\frac{1}{|R_\nu^*|}\int_{R_\nu^*} F(y) \dy \geq 
\frac{1}{|R_\nu^*|} 
\int_{R_\nu^*\setminus E} F(y) \dy \geq \\
\frac{|R_\nu^*\setminus E|}{|R_\nu^*|} \frac{F(x)}{A^{8M}} \geq 
 3^{-n}(1-2\omega_n10^{-2n})\frac{F(x)}{A^{8M}}\ .
\end{multline*}
The last inequality is true since $R_\nu$ contains
no $E^1$-bad boxes neither $E^2$-bad boxes.

Conversely, let $x\in R_\nu$ and let $J$ be the set of $\alpha$'s for which
$R_\alpha \subseteq R_\nu^*$.
For all $\alpha\in J$, let $x_\alpha\in R_\alpha\setminus (E^1\cup E^2)$.
Such points exist, since $R_\nu\not\subseteq B$.
Then,
\begin{multline*}
   \frac{1}{|R_\nu^*|}\int_{R_\nu^*} F(y) \dy = 
\frac{1}{|R_\nu^*|}\sum_{\alpha\in J} \int_{R_\alpha} F(y) \dy \leq \\
\frac{1}{|R_\nu^*|}\sum_{\alpha\in J} AF(x_\alpha)|R_\alpha^*| 
\leq 3^nA^{8M+1} F(x)\ .
\end{multline*}
\end{proof}

\section{Proof of the Lower Bound in Theorem~\ref{thm:nodal-tube}}
\label{sec:lower}
First, we prove the following 
proposition which is announced in the introduction
of~\cite{DF}.  
\begin{proposition}
\label{prop:many-good-balls} There exists a finite collection of balls
$B_i=B(x_i, r)$ centered at $x_i$ of radius $r=C_1/\mu$ which satisfy
\begin{itemize}
\item[(i)] $\phi_\mu(x_i) = 0$,
\item[(ii)] their doubles $2B_i=B(x_i, 2r)$ are pairwise
disjoint,
\item[(iii)] Denote by $B_i^+$ the set $\{\phi_\mu>0\}\cap B_i$.
Similarly, we define $B_i^-$.
Then 
$$ C_2<\frac{\Vol(B_i^+)}{\Vol(B_i^-)}<C_3\ , $$
\item[(iv)] $\sum_i \Vol(B_i) > C_4\Vol(M)$.
\end{itemize}
\end{proposition}
\begin{proof}
It is enough to prove the proposition in a coordinate neighborhood $V$.
It is well known that there exists a constant $C_5$ such that
every cube of side $C_5/\mu$ contains a zero of $\phi_\mu$
(see~\cite{Bru}).
We can decompose $V$ into small cubes $R_\nu$ whose 
side is of size $\delta=3C_5/\mu$. We call this subdivision $\Dl$.
Each cube $R_\nu$ contains a zero $x_\nu$ of $\phi_\mu$ in its middle third.
We now take a refinement $\Dl^1$ of $\Dl$:
We partition each side of a cube $R_\nu$ into $M$
intervals of equal sizes, where $\mu\delta/M$ is small enough
in order to apply Proposition~\ref{prop:eigen-small-scale-v2}. We deduce
that $(\phi_\mu^2, \Dl^1, E^1_\eps, C_6(\eps))$ is true 
(cf. definitions~\ref{def:approximate-average} \&~\ref{def:goodness}).
If $\Dl^2$ is the subdivision obtained by taking
the centers of cubes belonging to $\Dl^1$,
then the same proposition gives that
$(\phi_\mu^2, \Dl^2, E^2_\eps, C_6(\eps))$ is true.
Let $B$ be as in Proposition~\ref{prop:refinement},
and let $E=E^1_\eps\cup E^2_\eps\cup B$.
Then, $(\phi_\mu^2, \Dl, E, C_7(\eps))$ is true.

For each $E$-good cube $R_\nu$  
we pick a ball $B_\nu\subset R_\nu$ 
whose center is $x_\nu$ and whose radius $=\delta/6$.
By Proposition~\ref{prop:geometry-control} 
$$\frac{1}{C_8(\eps)}\leq \frac{\Vol(B_\nu^+)}{\Vol(B_\nu^-)} 
\leq C_8(\eps)\ .$$ 
The crucial point is to estimate the number of
$E$-good cubes. By Lemma~\ref{lem:bad-proportion},
the proportion of $E$-good cubes is
$\geq (1-C_9|E|)$ (which can be negative). 
It only remains to estimate $|E|$:
$|E^1_\eps| \leq C_{10}\eps\mu\delta$, $|E^2_\eps|\leq C_{11}\eps\mu\delta$
and 
$$|B|\leq \delta^n 3^n (\#(E^1_\eps)\mbox{-bad cubes}
+ \#(E^2_\eps)\mbox{-bad cubes}) \leq 
C_{12}(|E^1_\eps|+|E^2_\eps|)\leq C_{13}\eps\mu\delta\ .$$
So $|E|\leq C_{14}\eps\mu\delta$.
To conclude, we take $\eps$ small enough in order that the 
proportion of good cubes is $\geq 70\%$.   
\end{proof}
\begin{proof}[Proof of Theorem~\ref{thm:nodal-tube} - Lower Bound]
The next proposition gives a lower bound in a good ball.
\begin{proposition}
\label{prop:volume-in-good}
  Let $B(x, r)$ be one of the balls described above.
  Then we have $\Vol(T_{\mu, \delta} \cap 2B) \geq C_{15} r^{n-1}\delta$,
  whenever $\mu\delta<C_{16}$.
\end{proposition}
\begin{proof}
Let $(B^+)_\delta$ be a $\delta$-neighborhood of $B^+$, and
similarly for $(B^-)_\delta$. Since \mbox{$T_{\mu, \delta}\cap 2B$}
$\supseteq$  \mbox{$(B^+)_\delta\cap (B^-)_\delta$}, it is clear
that
$$
\Vol(T_{\mu, \delta}\cap 2B) \geq \Vol(B^+)_\delta +
\Vol(B^-)_\delta - \Vol (B(x, r+\delta))\ .
$$
Assume first that the metric $g$ is flat on $2B$.
By the Brunn-Minkowski Inequality~\cite[\S3.2.41]{federer} we know
$$\Vol(B^+)_\delta \geq \Vol(B^+) +
n\omega_n^{1/n}\delta\Vol(B^+)^{1-1/n}\ ,$$
where $\omega_n$ is the volume of the $n$-dimensional unit ball. We have the same
inequality for $(B^-)_\delta$. Set $\Vol(B^+) = \alpha\Vol(B)$, and
$\Vol(B^-) = (1-\alpha)\Vol(B)$. We have
\begin{multline}
\label{tube-in-2b}
\Vol(T_{\mu, \delta}\cap 2B) \geq \Vol(B) - \Vol(B(x, r+\delta)) + \\
n\omega_n^{1/n}\delta \Vol(B)^{1-1/n}
\left(\alpha^{1-1/n} + (1-\alpha)^{1-1/n}\right)\geq \\
\omega_n (r^n - (r+\delta)^n) + n\omega_n r^{n-1}\delta
\left(\alpha^{1-1/n}+(1-\alpha)^{1-1/n}\right)\ .
\end{multline}
At this point one observes that when $\alpha$ is bounded away from
$0$ and $1$ we have $\alpha^{1-1/n}+(1-\alpha)^{1-1/n} > 1 + C_{17}$.
So, if we take $\delta/r=C_{18}\mu\delta$ small enough then the last
expression in~(\ref{tube-in-2b}) is positive and we obtain
$$\Vol(T_{\mu, \delta}\cap 2B) \geq C_{19} n\omega_n r^{n-1}\delta\ .$$

Finally, since the metric $g$ is comparable to a flat metric on a small ball,
we have a similar inequality also for $g$.
\end{proof}

To conclude the proof of the lower bound in
Theorem~\ref{thm:nodal-tube} we observe that due to
Proposition~\ref{prop:many-good-balls}~(iv) the number of balls in
Proposition~\ref{prop:many-good-balls} is $>C_{20}\mu^n$. So,
by Proposition~\ref{prop:volume-in-good}
 $\Vol(T_{\mu, \delta}) > C_{21}\delta/\mu^{n-1}\cdot\mu^n 
= C_{22}\mu\delta\ .$
\end{proof}
\section{Proof of the Upper Bound in Theorem~\ref{thm:nodal-tube}}
\label{sec:upper}
In this section we estimate from above the volume of a tubular neighborhood
of the nodal set. The proof is based on the study in Section~\ref{sec:eigen-small-scale} of eigenfunctions in small scales.

Let $\cV=\{V_k\}$ be a covering of $M$ by small open sets.
Let $R_k\subseteq V_k$ be a cube preferred by
Proposition~\ref{prop:eigen-small-scale-v2}. The next lemma shows
that it is enough to estimate the volume of $T_{\mu, \delta}$ in
preferred cubes.
\begin{lemma}
\label{lem:covering}
There exists a covering $\cV=\{V_k\}$ on $M$ with the following properties
\begin{itemize}
  \item[(a)] $\cV$ is a finite covering.
  \item[(b)] the metric $g$ can be developed in power series in each chart $V_k$.
  \item[(c)] $M=\cup_k R_k$ for some choice of cubes $R_k\subseteq V_k$ preferred by
   Proposition~\ref{prop:eigen-small-scale-v2}.
\end{itemize}
\end{lemma}
\noindent We defer the proof of this Lemma to
Section~\ref{subsec:covering}.

Now, let $R\subseteq V$ be a preferred cube. We decompose it into
boxes $R_\nu$, where the sides of $R_\nu$ are 
of sizes $\delta<l^{(k)}_\nu <2\delta$.
We will denote this subdivision by $\Dl$.
\begin{definition}
We call $R_\nu$ a \emph{nodal} box if $\cN_\mu\cap R_\nu\neq\emptyset$.
\end{definition}
Let us denote the set of nodal boxes $R_\nu$ by
$\mathsf{Nod}$. Recall that $R_\nu^*$ denotes the union of $R_\nu$ with
its $3^n-1$ neighbors.
\begin{lemma}
\label{lem:tube-in-cubes}
  $T_{\mu, \delta}\subseteq\cup_{R_\nu\in\mathsf{Nod}} R_\nu^*$.
\end{lemma}
It remains to estimate the number of nodal boxes.
Fix $\eps=1$. Proposition~\ref{prop:eigen-small-scale-v2}
tells us that $(\phi_\mu^2, \Dl, E, C_1)$ is true 
(cf. Def.~\ref{def:approximate-average} \&~\ref{def:goodness}).
\begin{lemma}
\label{lem:num-good-nodal}
 The number of $E$-good nodal cubes is 
$\leq C_2\Vol_{n-1}(\cN_\mu)/\delta^{n-1}$.
\end{lemma}
\begin{proof}
We begin by
\begin{claim}
\label{clm:area-near-good}
  Let $R_\nu$ be an $E$-good nodal cube. Then
  \begin{equation}
   \label{area-near-good}
    \Vol_{n-1}(\cN_\mu\cap R_\nu^*) \geq C_3\delta^{n-1}\ .
  \end{equation}
\end{claim}
\begin{proof}[Proof of Claim]
First we see from the Brunn-Minkowski Inequality
as in Proposition~\ref{prop:volume-in-good}
that 
\begin{equation}
\label{minkowski}
  \liminf_{t\to0} \frac{\Vol(T_{\mu, t})}{t} \geq C_3\delta^{n-1}.
\end{equation}
Since $\cN_\mu$ is an analytic set, it is rectifiable
(\cite[Theorem 3.4.8 (13)]{federer}) and thus (\cite[Theorem 3.2.39]{federer}),
the limit in~(\ref{minkowski}) exists and equals 
$\Vol_{n-1}(\cN_\mu\cap R_\nu^*)$. 
\end{proof}

  Summing up~(\ref{area-near-good}) over all good nodal cubes we arrive at
  \begin{multline}
   3^n\Vol_{n-1}(\cN_\mu)\geq
  \sum_\nu \Vol_{n-1}(\cN_\mu\cap R_\nu^*) \geq \\
  \sum_{\mbox{good nodal }R_\nu\mbox{'s}} \Vol_{n-1}(\cN_\mu\cap R_\nu^*) \geq
  C_4\#(\mbox{good nodal }R_\nu\mbox{'s})\delta^{n-1}\ .
   \end{multline}
\end{proof}

\begin{proof}[Proof of Theorem~\ref{thm:nodal-tube}]
By Lemma~\ref{lem:bad-proportion} we know that 
the number of $E$-bad nodal cubes is $\leq
C_5\mu/\delta^{n-1}$. By Lemma~\ref{lem:num-good-nodal} and
Theorem~\ref{thm:nodal-volume} the number of $E$-good nodal cubes is $\leq
C_6\mu/\delta^{n-1}$. Together, we get that the number of nodal
cubes is $\leq C_7\mu/\delta^{n-1}$. By
Lemma~\ref{lem:tube-in-cubes}
$\Vol(T_{\mu, \delta}) \leq C_8\#(\Nod) \delta^n \leq C_{9}\mu\delta\ .$
\end{proof}

\subsection{Proof of Lemma~\ref{lem:covering}}
\label{subsec:covering}
The following lemma is clear by compactness of $M$.
\begin{lemma}
  There exists $\rho_0>0$ such that for all $p$, the metric $g$ can be
  developed in power series in $B(p, \rho_0)$.
\end{lemma}

Let $\rho_1 = \rho_0/C$ with $C$ large enough.
\begin{lemma}
\label{lem:preferred-in-center}
  Every ball $B(p, \rho_1)$ contains a preferred cube $R$ which contains $p$.
\end{lemma}
\begin{proof}
We identify $B(p, \rho)$ with the Euclidean ball $B(0, \rho)$ by
working in geodesic coordinates. Suppose that the point $x_0\in
R\subseteq B(0, \rho_0)$. Let $x_1\in B(0, \rho_0)$ with
$|x_1|=|x_0| =:r$. From proposition~\ref{prop:eigen-small-scale-v2} 
we know that $R$ is independent of $\mu$. 
By symmetry considerations, or just by
examining the proof of Proposition~\ref{prop:control-quotient} 
we see that any orthogonal transformation in $B(0, \rho_0)$ takes
$R$ to another preferred cube.

Now, given $p$, let $q$ be any point on $M$ such that $\dist(p, q)
= r$. The geodesic ball $B(q, \rho_0)$ contains a preferred cube
$R_1$ which contains $p$. Take a cube $R$ in
$R_1\cap B(p, \rho_1)$ which contains $p$.
\end{proof}

\begin{proof}[Proof of Lemma~\ref{lem:covering}]
By lemma~\ref{lem:preferred-in-center} we can cover $M$ by
preferred cubes. Then by compactness of $M$ we can extract a
finite covering by preferred cubes.
\end{proof}

\section{Approximation by Nodal Sets}
\label{sec:approx}
\begin{proof}[Proof of Theorem \ref{thm:approx}] The proof
proceeds similarly to the proof of Corollary \ref{ratls}.  Fix
$C,\eps>0$.  Let $T_{k, \delta}$ be the tubular  neighborhood of
$\cN(\phi_k)$ of radius $\delta_k=C/\mu_k^{n+1+\eps}$. By
Theorem~\ref{thm:nodal-tube} $\Vol(T_{k,\delta_k}) \leq
C/\mu_k^{n+\eps}$. We conclude that
\begin{equation}\label{union:ar}
\sum_k \Vol\left(T_{k,\delta_k}\right)\leq C\sum_k\mu_k^{-n-\eps}.
\end{equation}
By Weyl's Law~\cite{weyl,H} we know that \
$$
\mu_k\asymp  Ck^{1/n}.
$$
Hence
$$
\sum_k{\Vol}\left(T_{k,\delta_k}\right)\leq C\sum_k k^{-1-\eps/n}
$$
is finite. So, by the Borel-Cantelli Lemma (see e.g.~\cite{feller})
we obtain
$$\Vol(\cap_{j=1}^{\infty}\cup_{k=j}^{\infty} T_{k, \delta_k}) = 0\ .$$
\end{proof}

\section{Dimension two}
\label{sec:dim2}
\begin{theorem}
\label{thm:tube-dim2}
  Let $(\Sigma, g)$ be a smooth (i.e.\ $C^\infty$) closed Riemannian surface.
  Then there exist $C_1, C_2>0$ such that
  $$C_1\mu\delta\leq \Vol(T_{\mu, \delta}) \leq C_2\length(\cN_\mu)\delta\ .$$
\end{theorem}
\noindent In particular, Theorem~\ref{thm:nodal-tube} is true for
surfaces which satisfy Yau's conjecture.

We recall from~\cite{DF2} that for any smooth surface
$\length(\cN_\mu)\leq C_3\mu^{3/2}$. Hence, if we modify the  proof
of Theorem~\ref{thm:approx} according to Theorem~\ref{thm:tube-dim2}
we obtain
\begin{proposition}
\label{smooth2d}
Let $(\Sigma, g)$ be a closed compact surface with a {\em smooth} metric $g$.
Then we have
$\Vol(M(7/2+\eps,C))=0$ for all $C,\eps>0$.
\end{proposition}

\subsection{Lower Bound in Theorem~\ref{thm:tube-dim2}}
This is basically Br\"uning's argument. We can cover a fixed portion
of $\Sigma$ with pairwise disjoint balls $B_i= B(x_i, r)$ of radius
$r=c/\mu$ and such that $\phi_\mu(x_i)=0$. The set $\cN_\mu\cap
B(x_i, r)$ is of length $\geq r$. Moreover, in local coordinates it
has a projection of length $\geq cr$ on one of the axes. This
implies that $T_{\mu,\delta}\cap B(x_i, r)$ has area $\geq
cr\delta$. Summing up over all the balls $B_i$ we obtain
$$\Vol(T_{\mu, \delta}) \geq c_1\mu^2 \cdot c_2\delta/\mu =
c_3\mu\delta\ .$$

\subsection{Upper Bound in Theorem~\ref{thm:tube-dim2}- First Proof}
Let an eigenfunction $\phi_\mu$ have nodal domains
$\Omega_1,\ldots,\Omega_{N(\mu)}$. Given
$\bd\Omega_j\subset\cN_\mu$, let $L_j(t)$ denote the {\em interior
parallel} of $\bd\Omega_k$ at the distance $t$ inside $\Omega_j$.
It is clear that
\begin{equation}\label{vol:int}
\ar (A_\mu)=\sum_{j=1}^N\int_{t=0}^{\delta}{\rm length}(L_j(t)) \dt.
\end{equation}

The following inequality can be found in~\cite[Proposition A.1.iv]{S}:
\begin{equation}\label{savo}
{\rm length}(L_j(t))\leq {\rm length}(\bd\Omega_j)+
R(\Omega_j)\max\left\{\int_{\Omega_j}K^+ -
2\pi\chi(\Omega_j),0\right\}.
\end{equation}
Here $K^+$ denotes the positive part of the Gauss curvature,
$\chi(\Omega_j)$ is proportional to the number $m_j=m_j(\mu)$ of
connected components of $\bd\Omega_j$, and $R(\Omega_j)$ denotes the
inner radius of $\Omega_j$.  We substitute \eqref{savo} into
\eqref{vol:int} and sum over $1\leq j\leq N$.
By Proposition~\ref{prop:dense}  we know that
$R(\Omega_j)\leq C/\mu.$
We get the estimate
\begin{equation}\label{ar:est1}
\frac{\ar (A_\mu)}{\delta} \leq 2\cdot {\rm
length}(\cN_\mu) + \frac{C\int_M K^+}{\mu} + \frac{4\pi
C}{\mu}\sum_{j=1}^{N(\mu)} m_j(\mu)
\end{equation}
As $\mu_j=\mu\to\infty,$ the second term goes to zero.  It remains to
estimate the third term.  One can construct a connected graph
on $M$ whose edges will include all arcs of $\cN_\mu$, and show
using Euler's formula that
$$
\sum_{j=1}^N m_j\leq 2(N+g-1),
$$
where $g$ denotes the genus of the surface $M$.  Also, by Courant's
nodal domain theorem
$$
N=N(\mu_k)\leq k+1.
$$
We recall that by \cite{weyl,H} in dimension two $\mu_k\asymp
C\sqrt{k}$, hence $N(\mu_k)\leq C\mu_k^2$.  It follows that the
third term in the right-hand side of \eqref{ar:est1} is less than
$C\mu$.  Substituting everything back into \eqref{ar:est1} and recalling that
$\length(\cN_\mu)\geq C\mu$ (see~\cite{Bru}) we get
the desired estimate.
\qed

\subsection{Upper Bound in Theorem~\ref{thm:tube-dim2}- Second Proof}
\label{proof2d:2}

This proof was communicated by M.~Sodin.

It suffices to give a proof for the neighborhood of $\cN_\mu$ of size
$\delta/3$. We cover $M$ with cubes of
side $C\delta$ ({\em large cubes}), as well as by cubes of
side $C\delta/3$ ({\em small cubes}). One can easily
arrange that each cube intersects a bounded number of other cubes.
For every small cube, there exists a unique concentric large cube
whose side is three times larger.  To estimate the area of $T_{\mu,\delta}$, it
suffices to estimate the volume of the union $B_j$ of all small
cubes which intersect the nodal set $\cN_\mu$. Indeed, if
$x\in T_{\mu, \delta}$, then $\cN_\mu$ intersects either the small cube containing
$x$, or one of the $8$ neighboring small cubes, so the volume of
$T_{\mu,\delta}$ is at most $9\cdot \vol(B_j)$.

We distinguish several cases
\begin{itemize}
\item[i)] $\cN_\mu$ intersects a small cube $Q$, but any connected
component of $\cN_\mu\cap Q$ doesn't intersect the boundary of the big
concentric cube $Q'$.
\item[ii)] $\cN_\mu$ intersects a small cube $Q$, and there exists a connected
component of $\cN_\mu\cap Q$ that intersects the boundary of the big
concentric cube $Q'$.
\end{itemize}

In case (i)  there is at least one nodal domain contained in $Q'$,
so by the Faber-Krahn Inequality (see~\cite[Ch.~7, Th.~1]{egokon}) we get that
the area of this nodal domain is $> C/\mu^2$. By the Isoperimetric
Inequality, the length of $\cN_\mu\cap Q'$ is at least $C/\mu\geq C\delta$.

In case (ii), the length of $\cN_j\cap Q'$ is at least
$\delta/3$.

Hence, we conclude that the number
of $Q'$ for which $Q$ satisfies case~(i) or case~(ii) is $\ll\length(\cN_\mu)/\delta$.
Accordingly, the sum of the areas of those cubes is
\begin{equation}\label{2d:type2}
\ll\length(\cN_\mu)/\delta\cdot\delta^2 \leq C\length(\cN_\mu) \delta.
\end{equation}

\qed


\section{Discussion}
\label{sec:discussion}

For a given $M$ it seems interesting to find
$$
E(M):=\sup\{b:\vol(M(b,C))>0\ for\ some\ C>0\}.
$$
Theorem \ref{thm:approx} implies that on real-analytic
$n$-dimensional manifolds, $E(M)\leq n+1$. In dimension one, it
follows from the theory of continued fractions that $E(M)=2$ for
$M=[0,\pi]$.  In fact, $M(2,\pi)=M$ while
$\Vol(M(2+\eps,C))=0\ \forall \eps>0$.

The same result likely holds for {\em separable
systems} (Examples include surfaces of revolution,
Liouville tori and {\em quantum completely integrable} systems
\cite{TZ}). In such systems one can separate variables and choose a
basis of eigenfunctions that (in appropriate coordinates) have the
form $\phi(x_1,\ldots,x_n)=\prod\psi_j(x_j),$ where $\psi_j$ are
solutions of 2nd order differential equations. Accordingly,
$\cN(\phi)$ forms a ``grid'' of hypersurfaces determined by zeros of
$\psi_j$-s, and approximation by  $\cN(\phi)$ reduces to a series of
one-dimensional problems.

As a model example we consider an $n$-dimensional cube
$$
M(n)=\prod_{j=1}^n\; [0,\pi/\alpha_j],
$$
with Dirichlet boundary conditions, where for simplicity we assume
$\{\alpha_j^2\}_{j=1}^n$ are linearly independent over $\ratls$.
Then the eigenvalues have the form $\sum_{j=1}^n\alpha_j^2m_j^2$
(where $m_j\in\natls$) and are simple, while the corresponding
eigenfunctions have the form
$$
\phi(m_1,\ldots,m_n;x_1,\ldots,x_n)= \prod_{j:\,m_j\neq 0}
\sin(m_j\alpha_j x_j).
$$
\begin{proposition}\label{nbox}
$E(M(n))=2$ for all $n$.
\end{proposition}

\noindent{\bf Proof of Proposition \ref{nbox}.}

We first make a change of variables $y_j=\pi\alpha_jx_j$.  This
change of variables will only affect constants in the rate of
approximation by nodal sets; it won't affect the exponent.
In the rescaled coordinates, nodal sets have the form
\begin{equation}\label{nodal:box}
\cN(\phi(m_1,\ldots,m_n))=\cup_{j:m_j\neq 0}\cA_j,
\end{equation}
where
$\cA_j:=\{(y_1,\ldots,y_n):y_j=k_j/m_j,\ \ 0\leq k_j\leq m_j\}$.
We first show that
\begin{claim}
$E(M(n))\geq 2$.
\end{claim}

\begin{proof}
Let $(y_1,\ldots,y_n)\in M$ be an arbitrary point on $M$; we have
$0\leq y_j\leq 1$.  We can assume without loss of generality that
$y_j\notin\ratls,\forall 1\leq j\leq n$, since the set of such
points has the full measure. Consider next the continued fraction
expansion of its first (say) coordinate,
$$
y_1=[0;a_1,a_2,\ldots],
$$
where we use the notation of \cite{K}.  Let $p_k/q_k,\;
k=1,2,\ldots$ be the corresponding continued fractions.  Then the
points $(p_k/q_k,y_2,\ldots,y_n)\in\cN(\phi(q_k,0,\ldots,0))$, and
the Claim follows from the well-known inequality \cite{K}
$$
|y_1-p_k/q_k|<1/q_k^2.
$$
\end{proof}

We next show that
\begin{claim}
$E(M(n))\leq 2$.
\end{claim}

\begin{proof}
It suffices to show that $\Vol(M(2+\eps,C))=0$ for all $C,\eps>0$.
Let $\by=(y_1,\ldots,y_n)\in M(2+\eps,C)$.  As before, we may assume
that $y_j\notin\ratls$.  We know that there exists a sequence of
eigenvalues $\mu_k\to\infty$ such that
$d(\by,\cN(\phi_{\mu_k}))<C/\mu_k^{2+\eps}$.  Since all distances on
$[0,1]^n$ are equivalent, we may define $d(\bx,\by)=\max_{1\leq
j\leq n}|x_j-y_j|$.

In view of \eqref{nodal:box}, it follows that for some $1\leq j\leq
n$ (say, for $j=1$), there exists a sequence of integers
$q_k,k=1,2,\ldots$, such that $q_k\to\infty$ and
$|y_1-p_k/q_k|<C/q_k^{2+\eps}$ for some $0\leq p_k\leq q_k$.  The
Claim now follows from Corollary \ref{ratls}.  This also finishes
the proof of Proposition \ref{nbox}.
\end{proof}

For manifolds with ergodic geodesic flows (e.g.\ in negative
curvature), eigenfunction behavior has been studied using {\em
random wave model} \cite{Be}.  In addition, {\em percolation model}
\cite{BS} has been used to study the statistics of nodal domains in
chaotic systems.
We refer the reader to \cite{FGS} and
references therein for a nice discussion about applicability of
those models for studying various questions about eigenfunctions of
chaotic systems.

In the opinion of the authors, it would be difficult to use these
models directly to predict the ``best possible'' rate of
approximation by nodal sets.  The reason is that these models
describe a {\em single} eigenfunction on a scale of $C/\mu$ (several
wavelengths).  However (as shown by the example of $M=[0,\pi]$) for
a given $x\in M$ the values of $\mu$ giving the best approximation
of $x$ by $\cN(\phi_\mu)$ can grow exponentially. It thus seems
difficult to take into account simultaneous behavior of all
eigenfunctions in such a large energy range. However, one can
probably expect that $E(M)>2$ for such manifolds (in contrast to the
integrable case), due to irregularity of nodal lines for such
systems.

It also seems interesting to study ``level sets'' $M(b)$ for the approximation
exponent $b$, e.g. defined by
$$
M(b):=\cup_C M(b,C)\setminus \left(\cup_{a<b}\cup_C M(a,C)\right).
$$

\begin{remark} It should follow from the results of \cite{JL} that
the conclusion of Theorem \ref{thm:approx} should also hold for {\em
level sets} of eigenfunctions (since the level set of an
eigenfunction is a nodal set of a linear combination of that
eigenfunction with a constant eigenfunction). It seems interesting
to determine which level sets are $C/\mu$-dense (like nodal sets).
\end{remark}



\providecommand{\bysame}{\leavevmode\hbox to3em{\hrulefill}\thinspace}

\end{document}